\documentclass[a4paper,12pt]{article}
\usepackage{amssymb,amsthm,amsfonts}
\newtheorem{theorem}{Theorem}

\newcommand{\Z}{{\mathbb Z}}

\newcommand{\T}{{\mathbb T}}

\title{Inhomogeneous Partition Regularity}
\author{Imre Leader\footnote{Centre for Mathematical Sciences,
Wilberforce Road,
Cambridge CB3 0WB,
UK,  {\tt I.Leader@dpmms.cam.ac.uk}}\and 
Paul A.~Russell\footnote{Churchill College, Cambridge CB3 0DS, UK,
  {\tt P.A.Russell@dpmms.cam.ac.uk}}}

\begin{document}
\maketitle

\begin{abstract}
We say that the system of equations $Ax=b$, where $A$ is an integer matrix and
$b$ is a (non-zero) integer vector, is {\it partition regular} if whenever the
integers are finitely coloured there is a monochromatic vector $x$ with $Ax=b$.
Rado proved that the system $Ax=b$ is partition regular if and only if it has
a constant solution. 

Byszewski and Krawczyk asked if this remains true when the integers are replaced
by a general ring $R$. Our aim in this note is to answer this question in
the affirmative. The main ingredient is a new `direct' proof of Rado's result.
\end{abstract}

\begin{section}{Introduction}

The study of the Ramsey properties of inhomogeneous linear equations goes back to Rado
in the 1930s. Given an $m \times n$ matrix $A$, with integer entries, and a non-zero
vector $b \in \Z^m$, we say that the system of equations $Ax=b$ is {\it partition
regular}, or {\it partition regular over $\Z$}, if whenever the integers are finitely
coloured there is a vector $x \in \Z^n$, with all of its entries having the same colour,
such that $Ax=b$. 

There is a trivial sufficient condition for the system $Ax=b$ to be partition
regular, namely that it admits a constant solution, meaning an integer vector $x$ with all 
of
its entries the same. Rado \cite{R} showed that, in fact, this sufficient condition is also
necessary. We mention in passing that Rado also investigated the question of when there
is a monochromatic vector $x$ with all entries {\it positive} such that $Ax=b$, obtaining
necessary and suffiicent conditions in this case as well -- see \cite{R}. There is also a very
large literature on the homogeneous system $Ax=0$, also going back to \cite{R} -- see
Graham, Rothschild and Spencer \cite{GRS}.  

Bergelson, Deuber, Hindman and Lefmann \cite{BDHL} considered what happens in general 
(commutative) rings.
So now $A$ is an $m \times n$ matrix with entries from a ring $R$, and $b$ is a non-zero
element of $R^m$, and we say that the system $Ax=b$ is {\it partition regular over $R$}
if, whenever $R$ is finitely coloured, there is a monochromatic vector $x \in R^n$ with
$Ax=b$. The interesting question is: when this is equivalent to the condition 
that there is a constant solution to $Ax=b$? They showed that, as with $\Z$, this 
is indeed the case for a
class of integral domains, namely those integral domains that are not fields but have the
property that $R/(r)$ is finite for every non-zero $r \in R$. Here as usual $(r)$ denotes
the ideal generated by $r$, namely $\{tr: t \in R \}$.

This work was considerably extended by Byszewski and 
Krawczyk \cite{BK}. They showed that the 
result is true for every integral domain, and also for some other cases (such as for 
reduced rings, meaning rings with no nilpotent elements, satisfying a certain extra 
condition on their prime ideals). They asked if the result holds for all rings.
More generally, they also considered what happens if the vector $b$ has entries not from
$R$ but from some $R$-module $M$ (and we are finitely colouring $M$). They showed
that the condition that there is a constant solution is again equivalent, when $R$ is
an integral domain and $M$ is torsion-free, and also for any $R$-module $M$ when $R$ is
a Dedekind domain. Again, they asked if this condition is always equivalent (for any ring
$R$ and any $R$-module $M$) to the condition that there is a constant solution.

Our aim in this note is to answer these questions in the affirmative. The key new idea is
a `direct' approach to Rado's result. Rado himself proved his result by first proving it
for a single equation (i.e.~the case $m=1$), and then showing how one may pass from a 
single equation to the
case of several equations. The work in \cite{BDHL} proceeded along the same lines, as do the
results in \cite{BK} -- indeed, one of the most elegant results in \cite{BK} is 
that if we are dealing
with a single equation then the condition of partition regularity is always (for any $R$ and
$M$) equivalent to the existence of a constant solution. The work then comes (in \cite{BDHL} and
\cite{BK}) in the attempt to use this to build up to the case of many equations. 
Here, in contrast, we consider
the equations `all together'. We present a new very short proof of Rado's result that is
direct (in other words, not going via the case of a single
equation). And this proof, interpreted suitably, generalises to work 
for any ring $R$. 

The plan of the paper is as follows. In Section 2 we give our proof of Rado's result,
and in Section 3 we generalise to arbitrary rings (and also arbitrary modules over those
rings).
 
\end{section}

\begin{section}{A new proof of Rado's result}

For completeness we restate Rado's result.

\begin{theorem} (Rado \cite{R}). Let $A$ be an $m \times n$ integer matrix and let 
$b \in \Z^m$ be non-zero. Then the system of equations $Ax=b$ is partition regular if and
only if it has a constant solution.
\end{theorem}

\begin{proof} 

Let $c^{(1)},\ldots,c^{(n)}$ be the columns of $A$, and write $s$ for their sum. Suppose
that there is no constant solution: this means that $b$ does not belong to the
subgroup $H$ of $\Z^m$ generated by $s$.

Now, there must exist a group homomorphism $\theta$ from $\Z^m$ to a finite cyclic group
$\Z_d$ (the integers modulo $d$) such that $\theta(H)=0$ and $\theta(b) \neq 0$. Indeed, 
the quotient $G/H$ is a product of cyclic groups (as it is finitely generated), and in this
quotient the image of $b$ is non-zero, so there is a map to one of the cyclic factors that
does not kill $b$. If this cyclic factor is finite we are done, while if it is infinite
we compose with the projection from $\Z$ to $\Z_d$ for a suitable $d$.

Define an colouring of the integers with $d^n$ colours by colouring $t \in \Z$ with the
$n$-tuple $(\theta(c^{(1)}t),\ldots,\theta(c^{(n)}t))$. Suppose that for this
colouring we have a monochromatic vector $x$ with $Ax=b$. We have that
$\theta( \sum c^{(i)} x_i) = \theta (b)$. Since $\theta(s)=0$, we have 
$\theta( \sum c^{(i)} x_1) = 0$, and so combining these we have
$\sum (\theta(c^{(i)} x_i) - \theta(c^{(i)}x_1)) = \theta(b)$. But this is a 
contradiction, as each term in the sum on the left-hand side is zero by our choice of
colouring.

\end{proof}

The reader familiar with the result of Straus \cite{S} on colourings of abelian groups
will notice a similarity with the product colouring above (although in \cite{S} there is
no $\theta$ to worry about). Indeed, in Byszewski and Krawczyk \cite{BK} it is 
Straus's result that is applied to prove their result for a single equation. The main
difference is that we are working in $\Z^m$ instead of $\Z$. In the next section, it will
also be important that above we did not `reduce to the case when the column sum is zero',
because there is (in general) no ring homomorphism that does this: this is why we work with
the group homomorphism $\theta$ directly. In contrast, if one is dealing with a single
equation then (as shown in \cite{BK}) one may first pass to the case when the column sum
is zero, and then apply Straus's result itself. 

\end{section}

\begin{section}{The result for general rings}

We now turn to our main result.

\begin{theorem} Let $A$ be an $m \times n$ matrix with entries in a ring $R$, and let 
$b \in R^m$ be non-zero. Then the system of equations $Ax=b$ is partition regular over $R$ if and
only if it has a constant solution.
\end{theorem}

\begin{proof}

As before, let $c^{(1)},\ldots,c^{(n)}$ be the columns of $A$, and write $s$ for their sum. Suppose
that there is no constant solution: this means that $b$ does not belong to the
subgroup $H$ of $R^m$ consisting of all $rs$, $r \in R$.

We claim that there is a group homomorphism $\theta$ from $R^m$ to the circle $\T$ such that
$\theta(H)=0$ and $\theta(b) \neq 0$. This is a standard piece of group theory: let us choose a
subgroup $K$ of $R^m$ that contains $H$ and is maximal subject to not containing the element $b$. 
Then the quotient map $\theta$ from $G$ to $G/K$ 
has $\theta(b)$ non-zero and also every non-trivial subgroup of
$G/K$ must contain $\theta(b)$ (by the maximality of $K$). It follows from this that $G/K$ is
either a cyclic group of prime-power order or else the group $\Z_{p^\infty}$ of all $p^k$-th roots
of unity for any $k$ (for some fixed prime $p$). In each case this is a subgroup of $\T$.

Now define a colouring of $R$ with $d^n$ colours, where $d$ is a large positive integer, by colouring
$t \in R$ with the
$n$-tuple $(f(\theta(c^{(1)}t)),\ldots,f(\theta(c^{(n)}t)))$, where $f$ is the map that sends
the interval of the circle with arguments in $[2 \pi j /d, 2 \pi (j+1)/d)$ to $j$, each 
$0 \leq j \leq d-1$.  
Suppose that for this
colouring we have a monochromatic vector $x$ with $Ax=b$. We have that
$\theta( \sum c^{(i)} x_i) = \theta (b)$. Since $\theta(sx_1)=0$ (as $sx_1 \in H$), we have 
$\theta( \sum c^{(i)} x_1) = 0$, and so combining these we have
$\sum (\theta(c^{(i)} x_i) - \theta(c^{(i)}x_1)) = \theta(b)$. But this is a 
contradiction for $d$ large, as each term in the sum on the left-hand size has argument in
$[0, 4 \pi /d]$ by the definition of the 
colouring.

\end{proof}

Finally, note that the above proof goes through verbatim (replacing the abelian group
$R^m$ by the abelian group $M^m$) for $R$-modules.

\begin{theorem} Let $M$ be a module over a ring $R$. Let $A$ be an $m \times n$ matrix with entries in 
$R$, and let 
$b \in M^m$ be non-zero. Then the system of equations $Ax=b$ is partition regular over $M$ if and
only if it has a constant solution.
\end{theorem}

\end{section}

\end{document}